\documentclass[11pt]{article}
\usepackage[a4paper,margin=28mm]{geometry}
\usepackage[T1]{fontenc}
\usepackage[utf8]{inputenc}
\usepackage{lmodern}
\usepackage{microtype}
\usepackage{amsmath,amssymb,amsthm,mathtools}
\usepackage{enumitem}
\usepackage[round,authoryear]{natbib}
\usepackage[hidelinks]{hyperref}
\usepackage{url}

\newtheorem{theorem}{Theorem}
\newtheorem{lemma}[theorem]{Lemma}
\newtheorem{proposition}[theorem]{Proposition}
\newtheorem{corollary}[theorem]{Corollary}
\theoremstyle{remark}
\newtheorem{remark}[theorem]{Remark}

\DeclareMathOperator{\Hess}{Hess}
\DeclareMathOperator{\J}{J}
\newcommand{\Q}{\mathbb{Q}}
\newcommand{\C}{\mathbb{C}}
\newcommand{\id}{\operatorname{id}}

\title{Collision-Generated Compression for Homogeneous Keller Maps}
\author{Thomas Prellberg\\
\small School of Mathematical Sciences, Queen Mary University of London\\
\small Mile End Road, London E1 4NS, United Kingdom\\
\small \texttt{t.prellberg@qmul.ac.uk}}
\date{15 September 2026}

\begin{document}
\maketitle

\begin{abstract}
We associate to a collision $p\ne q$ of a homogeneous Keller map $F=\id+h$ the polarization subalgebra generated by $p$ and $q$. This collision hull is the smallest invariant linear subspace containing the collision; the restricted map remains a noninjective Keller map, and its dimension controls that of the standard symmetric lift. The construction is compatible with scalar extension and equivariant under linear conjugacy, so it gives a canonical linear carrier for a marked failure of injectivity. We compute this carrier exactly in two recent cubic-homogeneous reductions of the three-variable Jacobian counterexample. For Thompson's 24-variable map, the growth $2,4,11,20,20$ recovers MacFarlane's 20-dimensional invariant subspace. For the 19-variable homogenization of Van Rijn's 12-variable degree-three map, the growth $2,4,11,19,19$ fills the whole space. Hence no invariant linear restriction retaining the displayed collision can improve the corresponding 38-variable symmetric lift. We give the resulting 340-monomial homogeneous quartic explicitly: it is Hessian-nilpotent, violates Zhao's Vanishing Conjecture, and its gradient Keller map has an exact collision over $\Q(i)$. The analogous 20-variable application gives the previously studied 40-variable, 350-monomial quartic. All finite calculations are checked by exact companion code.
\end{abstract}

\noindent\textbf{2020 Mathematics Subject Classification.} Primary 14R15; Secondary 13N15.

\medskip
\noindent\textbf{Keywords.} Jacobian Conjecture; Vanishing Conjecture; Keller map; polarization algebra; invariant subspace; Hessian-nilpotent polynomial.

\section{Introduction and statement of the result}

For coordinates $z=(z_1,\ldots,z_n)$, write
\[
\Delta=\sum_{j=1}^n \frac{\partial^2}{\partial z_j^2}.
\]
Zhao proved that the all-dimensional Jacobian Conjecture is equivalent to the following quartic vanishing statement: if a homogeneous quartic $P\in\C[z_1,\ldots,z_n]$ satisfies
\[
\Delta^m P^m=0\qquad(m\ge 1),
\]
then $\Delta^mP^{m+1}=0$ for all sufficiently large $m$ \citep{Zhao2007}. He also identified the hypothesis with nilpotency of $\Hess P$ and gave an inversion formula involving the polynomials $\Delta^mP^{m+1}$. Consequently, a homogeneous Hessian-nilpotent quartic for which $\id-\nabla P$ is noninjective gives a counterexample to this precise form of the Vanishing Conjecture.

The underlying three-variable counterexample was announced by Alp\"oge, who credited Fable for the work leading to the example \citep{Alpoge2026}. Its geometry has already acquired several complementary descriptions. Speyer gave an explanation in terms of a tangent sweep \citep{Speyer2026}, Gao gave a self-contained account and further development of that mechanism \citep{Gao2026}, and van Dobben de Bruyn subsequently gave a coordinate-free interpretation through divisors in projective bundles whose complements are affine space \citep{vDdB2026}. These works clarify the geometry of the low-dimensional source. The question addressed here is different and reduction-theoretic: once a homogeneous Keller counterexample and a marked collision are given, what is the smallest linear carrier that must survive if that collision is to survive?

Thompson published an explicit 24-variable cubic-homogeneous Keller map with a rational collision \citep{Thompson2026}. MacFarlane then found four linear invariants and restricted it to an explicit 20-variable cubic-homogeneous Keller map retaining that collision \citep{MacFarlane2026}. The map displayed in Section~\ref{sec:twenty} below is exactly MacFarlane's map, after relabelling its coordinates.

The symmetric construction of de Bondt and van den Essen associates to a cubic-homogeneous map with nilpotent nonlinear Jacobian a homogeneous quartic in twice as many variables \citep{deBondtEssen2005a}. Santib\'a\~nez-Leal applied it to Thompson's 24-variable map, obtaining an explicit 48-variable quartic with 382 monomials and an exact gradient collision \citep{Santibanez2026}. Van Rijn recorded that applying the cotangent, equivalently symmetric, lift to MacFarlane's 20-variable map gives a 40-variable homogeneous quartic \citep{vanRijnStable2026}; a separate audit excludes constant-kernel quotients and invariant-hyperplane restrictions of that map \citep{vanRijnG20Audit2026}. Subsequently, a source--target coordinate-pair reduction produced a 12-variable degree-three Keller map whose cubic-output rank is six; rank-compressed homogenization gives a 19-variable cubic-homogeneous map and hence a 38-variable symmetric lift \citep{vanRijn12var2026}.

For a homogeneous map $h$, its full polarization equips the underlying vector space with a symmetric multiary algebra, usually called the polarization algebra of $h$ \citep{Umirbaev2019}. The stable space considered below is therefore the polarization subalgebra generated by the two collision points; we do not claim novelty for this generated-subalgebra construction itself. The contribution is that for a collision this subalgebra is a canonical carrier of noninjectivity: it supports a noninjective Keller restriction, is invariant under change of linear coordinates, and determines the dimension of the unchanged symmetric lift.

We apply this observation to two successive cubic-homogeneous reductions. In Thompson's original 24-dimensional space, the exact growth profile $2,4,11,20,20$ canonically recovers MacFarlane's invariant subspace. In Van Rijn's 19-dimensional map, the profile $2,4,11,19,19$ fills the entire space. The latter calculation shows that the 38-variable lift cannot be improved by any invariant linear restriction retaining its displayed collision, in any codimension. This is a route-specific obstruction rather than a claim of global minimality.

The construction has also entered an independent computational pipeline. The \texttt{kellermap} project is a Python library for certified transformations of polynomial Keller maps. Its current development line implements collision-generated compression as a certified \texttt{CompressionStep}, transports collisions through reduction chains, and implements the symmetric quartic lift as a \texttt{SymmetricLiftStep}; its documented 23-to-19-to-38 branch reproduces the same collision-hull dimension used here \citep{KellerMap2026}. The project also treats the stronger multi-affine normal form as a separate reduction branch rather than a prerequisite for the symmetric lift. This computational uptake is independent of the explicit 38- and 40-variable calculations below and illustrates the intended role of the collision hull as a composable reduction step.

Recent work also sharpens the low-dimensional boundary of related Hessian questions. Ni proves the quartic Hessian Conjecture in four variables \citep{Ni2026}, while Liu proves several positive results for the four-dimensional Hessian Conjecture in Lorentzian signature, including the degree-at-most-five case \citep{Liu2026}. These concern the constant-Hessian problem rather than the homogeneous Hessian-nilpotent quartics considered here. Likewise, the five-variable example of Meng and Yang has constant nonzero Hessian determinant and noninjective gradient \citep{MengYang2026}, and Dvorsky's five-variable counterexample to the Generalized Vanishing Conjecture uses a third-order constant-coefficient operator rather than the Laplacian \citep{Dvorsky2026}. Their variable counts are therefore not dimension comparisons with the quartic problem treated here.

Our main result combines the general compression theorem with two exact applications.

\begin{theorem}\label{thm:main}
The collision-generated polarization subalgebras of the two cubic-homogeneous Keller maps considered below have the following exact dimensions.
\begin{enumerate}[label=\arabic*.]
\item For Thompson's 24-variable map and his displayed rational collision, the stable subalgebra is exactly MacFarlane's 20-dimensional invariant subspace, with growth
\[
2,4,11,20,20.
\]
\item For the 19-variable cubic-homogeneous map obtained from Van Rijn's 12-variable degree-three Keller map, the displayed collision generates the whole 19-dimensional space, with growth
\[
2,4,11,19,19.
\]
\end{enumerate}
Consequently, the unchanged symmetric lifts give homogeneous quartics $P_{40}$ in 40 variables and $P_{38}$ in 38 variables, both over $\Q(i)$, with respectively 350 and 340 monomials. Each Hessian is nilpotent, and each polynomial satisfies
\[
\Delta^mP^m=0\quad(m\ge1),\qquad
\Delta^mP^{m+1}\ne0\quad\text{for infinitely many }m,
\]
and each gradient Keller map $\id-\nabla P$ has an explicit two-point collision over $\Q(i)$. Over each of $\Q$, $\Q(i)$, and $\C$, the lift dimensions 40 and 38 are minimal within the respective procedures that first restrict the given cubic-homogeneous map to an invariant linear subspace containing the displayed collision and then apply the symmetric lift without further alteration. No global minimality is asserted.
\end{theorem}

\section{Collision-generated compression}\label{sec:compression}

We first isolate the general mechanism used below. Recall that a polynomial map is Keller if its Jacobian determinant is a nonzero constant. For a homogeneous map $h$ with full polarization $T$, we use ``polarization subalgebra'' for a linear subspace closed under $T$, in the standard polarization-algebra terminology \citep{Umirbaev2019}.

\begin{lemma}[Restriction to an invariant linear subspace]\label{lem:restriction}
Let $k$ be a field, let $F\colon k^n\to k^n$ be a polynomial map with $\det\J F\in k^\times$, and let $W\subseteq k^n$ be a linear subspace whose defining ideal satisfies $F^*I(W)\subseteq I(W)$. Then $F|_W\colon W\to W$ is Keller.
\end{lemma}

\begin{proof}
Choose linear coordinates $(u,v)\in k^r\times k^{n-r}$ in which $W=\{v=0\}$, and write $F=(A,B)$. The assumed scheme-theoretic invariance gives $B(u,0)=0$, so
\[
\J F(u,0)=
\begin{pmatrix}
\J_uA(u,0)&\J_vA(u,0)\\
0&\J_vB(u,0)
\end{pmatrix}.
\]
The upper-left block is the Jacobian matrix of $F|_W$. Hence
\[
\det\J(F|_W)(u)\,\det\J_vB(u,0)=\det\J F\in k^\times.
\]
Both factors lie in $k[u_1,\ldots,u_r]$, whose only units are the nonzero constants. Thus $\det\J(F|_W)\in k^\times$.
\end{proof}

\begin{theorem}[Collision-hull compression and symmetric lift]\label{thm:collision-hull}
Let $k\subseteq\C$ be a field of characteristic zero, let $h\colon k^n\to k^n$ be homogeneous of degree $d\ge2$, and suppose that $F=\id+h$ is Keller. Let $p,q\in k^n$, with $p\ne q$, satisfy $F(p)=F(q)$. Let $T\colon(k^n)^d\to k^n$ be the symmetric $d$-linear polarization of $h$, so that $h(z)=T(z,\ldots,z)$, and define
\begin{equation}\label{eq:collision-hull}
W_0=\operatorname{span}_k\{p,q\},\qquad
W_{\nu+1}=W_\nu+\operatorname{span}_k\{T(w_1,\ldots,w_d):w_j\in W_\nu\}.
\end{equation}
Let $W$ be the stable value and put $r=\dim_kW$. Then:
\begin{enumerate}[label=\arabic*.]
\item $W$ is the smallest linear subspace containing $p,q$ and satisfying $h(W)\subseteq W$.
\item After choosing a basis of $W$ and retaining the notation $p,q$ for the corresponding coordinate vectors, the restriction is $g=\id+\bar h\colon k^r\to k^r$, where $\bar h$ is homogeneous of degree $d$, $\det\J g=1$, $\J\bar h$ is nilpotent, and the collision survives.
\item Over $K=k(i)$, put
\begin{equation}\label{eq:symmetric-lift}
P_W(x,y)=i\sum_{j=1}^r y_j\bar h_j(x+iy).
\end{equation}
Then $P_W$ is homogeneous of degree $d+1$, its Hessian is nilpotent, and $\id-\nabla P_W$ is a noninjective Keller map. More precisely, if
\[
\rho=(I+\J\bar h(q)^T)^{-1}(p-q),
\]
then $(p,0)$ and $(q+\rho,i\rho)$ have the same image.
\item If $d=3$, then $P_W$, viewed over $\C$, is a homogeneous counterexample to the quartic case of Zhao's Vanishing Conjecture:
\[
\Delta^mP_W^m=0\quad(m\ge1),\qquad
\Delta^mP_W^{m+1}\ne0\quad\text{for infinitely many }m.
\]
\end{enumerate}
The construction of $W$ commutes with extension of the ground field and is equivariant under linear conjugacy: if $A\in\operatorname{GL}_n(k)$, $F^A=A^{-1}FA$, and the marked collision is $(A^{-1}p,A^{-1}q)$, then its collision hull is $A^{-1}W$.
\end{theorem}

\begin{proof}
The sequence in \eqref{eq:collision-hull} stabilizes because its dimensions are nondecreasing and bounded by $n$. At the stable value, $T(W,\ldots,W)\subseteq W$, hence $h(W)\subseteq W$. Conversely, if $U$ contains $p,q$ and satisfies $h(U)\subseteq U$, the polarization identity
\[
T(v_1,\ldots,v_d)=\frac1{d!}\sum_{S\subseteq\{1,\ldots,d\}}(-1)^{d-|S|}
 h\left(\sum_{j\in S}v_j\right)
\]
shows that $T(U,\ldots,U)\subseteq U$. Induction gives $W_\nu\subseteq U$ for every $\nu$, proving minimality.

Since $k\subseteq\C$ is infinite, the pointwise inclusion $h(W)\subseteq W$ implies scheme-theoretic invariance: for every linear form $\ell$ vanishing on $W$, the polynomial $\ell\circ F$ vanishes on every $k$-point of $W$, hence belongs to $I(W)$. Thus $F^*I(W)\subseteq I(W)$, and Lemma~\ref{lem:restriction} shows that $g=F|_W$ is Keller. Since $\bar h$ has degree at least two, $\J g(0)=I$, so $\det\J g=1$. The original collision lies in $W$ and survives.

Put $A(z)=\J\bar h(z)$. Homogeneity gives
\[
1=\det\J g(tz)=\det(I+t^{d-1}A(z)).
\]
The homomorphism $k[z_1,\ldots,z_r,s]\to k[z_1,\ldots,z_r,t]$ defined by $s\mapsto t^{d-1}$ and fixing the $z_j$ is injective, so $\det(I+sA(z))=1$. Thus the characteristic polynomial of $A$ is $X^r$, and Cayley--Hamilton in $\operatorname{Mat}_r(k[z])$ gives $A^r=0$.

For the symmetric lift \citep{deBondtEssen2005a}, put
\[
f(x,y)=-i\sum_{j=1}^r y_j\bar h_j(x+iy),\qquad P_W=-f,
\]
and let $\mathcal F=\id+\nabla f=\id-\nabla P_W$. For
\[
S(a,b)=(a-ib,b),\qquad L(c,e)=(c+ie,e),
\]
direct differentiation gives
\begin{equation}\label{eq:triangular-conjugate}
L\circ\mathcal F\circ S(a,b)=
\bigl(a+\bar h(a),\ b+\J\bar h(a)^Tb-i\bar h(a)\bigr).
\end{equation}
The Jacobian matrix on the right is block triangular. Since $\det(I+\J\bar h(a))=1$, it has determinant
\[
\det(I+\J\bar h(a))\det(I+\J\bar h(a)^T)=1.
\]
Both linear maps have determinant one, so $\det(I+\Hess f)=1$. Since $\Hess f$ is homogeneous of degree $d-1$, scaling the variables and using the injective homomorphism
\[
K[x,y,s]\longrightarrow K[x,y,t],\qquad s\longmapsto t^{d-1},
\]
which fixes $x,y$, gives $\det(I+s\Hess f)=1$. Cayley--Hamilton now proves that $\Hess f$, and hence $\Hess P_W$, is nilpotent.

Since $g(p)=g(q)$, one has $\bar h(p)-\bar h(q)=q-p$. The matrix $I+\J\bar h(q)^T$ is invertible, and for $w=i\rho$,
\[
(I+\J\bar h(q)^T)w=i(p-q)=-i(\bar h(p)-\bar h(q)).
\]
It follows from \eqref{eq:triangular-conjugate} that the right-hand side takes the same value at $(p,0)$ and $(q,w)$. Since $L$ is invertible, this implies $\mathcal F(S(p,0))=\mathcal F(S(q,w))$. Finally, $S(p,0)=(p,0)$ and $S(q,w)=(q+\rho,i\rho)$, proving the collision. The two points are distinct: equality would imply $i\rho=0$, hence $\rho=0$, and then equality of the first coordinates would give $p=q$.

When $d=3$, Zhao's Hessian-nilpotence criterion \citep[Theorem~4.3]{Zhao2007} gives $\Delta^mP_W^m=0$ for all $m\ge1$. His inversion formula \citep[Theorem~3.4]{Zhao2007} gives the potential of the formal inverse of $z-t\nabla P_W$ as
\begin{equation}\label{eq:zhao-inverse}
Q_t=\sum_{m=0}^{\infty}\frac{t^m}{2^m m!(m+1)!}\,\Delta^mP_W^{m+1}.
\end{equation}
Eventual vanishing would make $Q_t$ polynomial in $t$ and the variables. Zhao's formula identifies the formal inverse with $z+t\nabla Q_t$. The formal composition identities would then be polynomial identities, so evaluating them at $t=1$ would give a polynomial inverse of $\id-\nabla P_W$, contrary to the collision. Thus the sequence is not eventually zero, or equivalently $\Delta^mP_W^{m+1}\ne0$ for infinitely many $m$.

Finally, multilinearity and induction on $\nu$ show that for every field extension $E/k$, the corresponding spaces satisfy $W_{\nu,E}=E\otimes_kW_\nu$. For linear conjugacy, if $h^A=A^{-1}hA$, then its full polarization is
\[
T^A(v_1,\ldots,v_d)=A^{-1}T(Av_1,\ldots,Av_d).
\]
Starting from $W^A_0=A^{-1}W_0$, induction in \eqref{eq:collision-hull} gives $W^A_\nu=A^{-1}W_\nu$ for every $\nu$, and hence $W^A=A^{-1}W$.
\end{proof}

\begin{remark}\label{rem:carrier}
The stable space $W$ is the polarization subalgebra generated by $p,q$ in the sense of \citet{Umirbaev2019}. The restriction lemma and symmetric lift are elementary or standard. The additional point of Theorem~\ref{thm:collision-hull} is that, when $p,q$ form a collision, this generated subalgebra supports a noninjective Keller restriction and therefore governs the dimension of the associated symmetric lift. Its equivariance under linear conjugacy makes its dimension, and the corresponding route-specific lower bound for the lift, independent of the chosen linear coordinates.
\end{remark}

\section{The twenty-variable cubic map}\label{sec:twenty}

Write Thompson's cubic-homogeneous map as $u+H(u)$ in 24 variables \citep{Thompson2026}. MacFarlane observed the termwise identities
\begin{equation}\label{eq:four-relations}
H_7=3H_5,\qquad H_8=H_6,\qquad H_{12}=H_{10},\qquad H_{17}=-2H_{15},
\end{equation}
and that both points of Thompson's collision lie in the invariant subspace
\begin{equation}\label{eq:V}
V=\{u_7=3u_5,\ u_8=u_6,\ u_{12}=u_{10},\ u_{17}=-2u_{15}\}\simeq\Q^{20}.
\end{equation}
This gives MacFarlane's 20-variable cubic-homogeneous restriction \citep{MacFarlane2026}.

Use the coordinates
\begin{equation}\label{eq:zcoords}
\begin{split}
(z_1,\ldots,z_{20})={}&(u_1,u_2,u_3,u_4,u_5,u_6,u_9,u_{10},u_{11},u_{13},\\
&u_{14},u_{15},u_{16},u_{18},u_{19},u_{20},u_{21},u_{22},u_{23},u_{24})
\end{split}
\end{equation}
on $V$. The inverse coordinate projection $\iota\colon\Q^{20}\to V\subseteq\Q^{24}$ is explicitly
\begin{equation}\label{eq:iota}
\begin{split}
\iota(z)=(&z_1,z_2,z_3,z_4,z_5,z_6,3z_5,z_6,z_7,z_8,z_9,z_8,\\
&z_{10},z_{11},z_{12},z_{13},-2z_{12},z_{14},z_{15},z_{16},z_{17},z_{18},z_{19},z_{20}).
\end{split}
\end{equation}
After the relabelling, MacFarlane's map is $g(z)=z+h(z)$, where
\begin{align*}
h_1={}&-z_{11}z_{12}z_{20}-z_{17}z_{20}^2-\tfrac32z_{19}z_{20}^2,\\
h_2={}&3z_1z_3z_{20}+3z_{14}z_{20}^2-3z_5z_6z_{20}-z_8z_9z_{20},\\
h_3={}&-z_8z_{10}z_{20}+z_{15}z_{20}^2+3z_{18}z_{20}^2+4z_2^2z_{20}-z_4z_5z_{20}-z_6z_7z_{20},\\
h_4={}&2z_{12}z_{13}z_{20}-z_{16}z_{20}^2,\\
h_5={}&z_{18}z_{20}^2+3z_2^2z_{20},\\
h_6={}&z_{19}z_{20}^2,\\
h_7={}&3z_2z_3z_{20}-z_2z_5z_{20},\\
h_8={}&z_1z_2z_{20},\\
h_9={}&6z_1z_3z_{20}-3z_1z_5z_{20}-3z_3z_6z_{20},\\
h_{10}={}&-z_1z_7z_{20}+7z_2^2z_{20}-z_3z_4z_{20},\\
h_{11}={}&z_1z_3z_{20},\\
h_{12}={}&-\tfrac12z_1^2z_{20},\\
h_{13}={}&z_2^2z_{20},\\
h_{14}={}&-4z_1z_2^2+\tfrac13z_1z_2z_9+2z_1z_3z_8-z_1z_5z_8+3z_2^2z_6-z_3z_6z_8,\\
h_{15}={}&z_1z_2z_{10}-z_1z_7z_8+3z_2^2z_4+7z_2^2z_8+3z_2z_3z_6-z_2z_5z_6-z_3z_4z_8,\\
h_{16}={}&-z_1^2z_{13}+2z_2^2z_{12},\\
h_{17}={}&\tfrac12z_1^2z_{11}-z_1z_3z_{12},\\
h_{18}={}&-z_1z_2z_3,\\
h_{19}={}&-z_1^2z_2,\\
h_{20}={}&0.
\end{align*}
Direct substitution in Thompson's 24 displayed components gives the ambient intertwining identity
\begin{equation}\label{eq:intertwine}
H\circ\iota=\iota\circ h,\qquad
(\id+H)\circ\iota=\iota\circ(\id+h).
\end{equation}
Define
\begin{equation}\label{eq:p}
p=(0,0,-\tfrac14,0,0,0,0,0,0,0,0,0,0,0,0,0,0,0,0,1),
\end{equation}
and
\begin{equation}\label{eq:q}
\begin{split}
q=(&1,-\tfrac32,\tfrac{13}{2},-\tfrac94,3,\tfrac32,\tfrac{99}{4},\tfrac32,-\tfrac34,-\tfrac{45}{8},-\tfrac{13}{2},\tfrac12,-\tfrac94,-\tfrac{15}{8},\\
&\tfrac{567}{16},-\tfrac92,\tfrac{13}{2},-\tfrac{39}{4},-\tfrac32,1).
\end{split}
\end{equation}
The corresponding points $\widetilde p=\iota(p)$ and $\widetilde q=\iota(q)$ are the two points of Thompson's collision lying in $V$.

\begin{proposition}\label{prop:twenty-map}
The map $g=\id+h\colon\C^{20}\to\C^{20}$ satisfies
\[
g(p)=g(q)=p,\qquad (\J h)^{17}\ne0,\qquad(\J h)^{18}=0.
\]
Moreover, for $\widetilde g=\id+H\colon\C^{24}\to\C^{24}$,
\[
\widetilde g(\widetilde p)=\widetilde g(\widetilde q)=\widetilde p.
\]
In particular, both maps are noninjective Keller maps.
\end{proposition}

\begin{proof}
Direct substitution gives the displayed 20-dimensional collision, and \eqref{eq:intertwine} gives the ambient collision. Exact multiplication of the polynomial matrix $\J h$ leaves six nonzero terms in its seventeenth power and no terms in its eighteenth power. Nilpotence gives $\det(I+\J h)=1$; Thompson's map is Keller by construction \citep{Thompson2026}.
\end{proof}

\section{Thompson's collision hull and the 40-variable lift}\label{sec:forty}

Let $T\colon(\Q^{20})^3\to\Q^{20}$ be the symmetric trilinear polarization of $h$, and form the collision hull
\[
W_0=\operatorname{span}_{\Q}\{p,q\},\qquad
W_{\nu+1}=W_\nu+\operatorname{span}_{\Q}\{T(a,b,c):a,b,c\in W_\nu\}.
\]
Exact rational row reduction gives
\begin{equation}\label{eq:growth20}
\dim W_0,\ldots,\dim W_4=2,4,11,20,20.
\end{equation}

\begin{proposition}[Canonical recovery of the 20-dimensional restriction]\label{prop:canonical20}
The whole space $\Q^{20}$ is the smallest linear subspace containing $p,q$ and invariant under $h$. In Thompson's original 24-dimensional space, the collision-generated invariant hull of $\widetilde p,\widetilde q$ under $H$ is exactly the subspace $V$ in \eqref{eq:V}. Both statements remain true after extension to $\Q(i)$ or $\C$.
\end{proposition}

\begin{proof}
Equation~\eqref{eq:growth20} and Theorem~\ref{thm:collision-hull} give the first assertion over $\Q$. Let $\widetilde W\subseteq\Q^{24}$ be the collision hull generated by $\widetilde p,\widetilde q$ for Thompson's nonlinear part $H$. Since $V$ is $H$-invariant and contains both points, minimality gives $\widetilde W\subseteq V$. Its pullback $\iota^{-1}(\widetilde W)$ contains $p,q$. It is $h$-invariant by \eqref{eq:intertwine}: if $z\in\iota^{-1}(\widetilde W)$, then $\iota(h(z))=H(\iota(z))\in\widetilde W$.

The first assertion therefore forces $\iota^{-1}(\widetilde W)=\Q^{20}$. Hence $\widetilde W=V$. The base-change assertion in Theorem~\ref{thm:collision-hull} gives the statements over the larger fields.
\end{proof}

\begin{corollary}[Route-specific minimality]\label{cor:min40}
Let $k\in\{\Q,\Q(i),\C\}$, and base-change Thompson's map, the subspace $V$, and the collision to $k$. Forty is minimal among examples obtained by restricting the resulting 24-variable map to a $k$-linear invariant subspace containing $\widetilde p,\widetilde q$ and then applying \eqref{eq:symmetric-lift} without further alteration.
\end{corollary}

\begin{proof}
Every such invariant subspace contains the collision hull $V$, so it has dimension at least twenty. The symmetric construction uses two copies of the cubic space.
\end{proof}

We now apply the lift with $r=20$:
\begin{equation}\label{eq:P40}
\boxed{P_{40}(x,y)=i\sum_{j=1}^{20}y_jh_j(x+iy).}
\end{equation}
Collection in the monomial basis gives 350 nonzero terms.

For completeness, the collision from Theorem~\ref{thm:collision-hull} is displayed explicitly. Put
\begin{equation}\label{eq:rho40-def}
\rho_{40}=(I+\J h(q)^T)^{-1}(p-q)\in\Q^{20}.
\end{equation}
The inverse exists because $\J h(q)$ is nilpotent. In coordinates,
\[
\begin{split}
\rho_{40}=(&-\tfrac{582383}{512},\ \tfrac{138141}{256},\ \tfrac{604023}{64},\ -\tfrac{26433}{64},\ \tfrac{17457}{64},\\
&-\tfrac{53763}{64},\ -\tfrac{2007}{16},\ \tfrac{25263}{128},\ \tfrac{27}{16},\ -\tfrac{1521}{32},\\
&\tfrac{39}{4},\ -63,\ \tfrac{27}{4},\ -\tfrac{413943}{256},\ -\tfrac{606291}{64},\\
&-\tfrac{26145}{64},\ -\tfrac{585711}{512},\ -\tfrac{914451}{32},\ -\tfrac{885405}{1024},\ \tfrac{169983}{64}).
\end{split}
\]
Thus the distinct points
\begin{equation}\label{eq:collision40}
A_{40}=(p,0),\qquad B_{40}=(q+\rho_{40},i\rho_{40})
\end{equation}
have the same image under $\id-\nabla P_{40}$. By Theorem~\ref{thm:collision-hull}, $P_{40}$ is Hessian-nilpotent and violates Zhao's quartic Vanishing Conjecture. As a direct low-order check, $\Delta(P_{40}^2)\ne0$, with 8,630 monomials.

\section{A nineteen-variable cubic map and its 38-variable lift}\label{sec:thirtyeight}

We now apply the same construction to the smaller cubic-homogeneous map obtained subsequently in \citet{vanRijn12var2026}. Write $a=(a_1,\ldots,a_{12})$. Van Rijn's degree-three map is
\[
K(a)=a+Q(a)+C(a),
\]
where the quadratic part is
\begin{align*}
Q_1&=-a_{11}a_{12}, & Q_2&=3a_1a_3-3a_5a_6-a_8a_9,\\
Q_3&=-a_{10}a_8+4a_2^2-a_4a_5-a_6a_7, & Q_4&=-a_8^2,\\
Q_5&=3a_2^2, & Q_6&=0,\\
Q_7&=3a_2a_3-a_2a_5, & Q_8&=a_1a_2,\\
Q_9&=6a_1a_3-3a_1a_5-3a_3a_6, & Q_{10}&=-a_1a_7+7a_2^2-a_3a_4,\\
Q_{11}&=a_1a_3, & Q_{12}&=-\tfrac12a_1^2,
\end{align*}
and the cubic part has $C_7=\cdots=C_{12}=0$ and
\begin{align*}
C_1={}&\tfrac12a_1^2a_{11}-\tfrac32a_1^2a_2-a_1a_{12}a_3,\\
C_2={}&12a_1a_2^2-a_1a_2a_9-6a_1a_3a_8+3a_1a_5a_8-9a_2^2a_6+3a_3a_6a_8,\\
C_3={}&-a_1a_{10}a_2+3a_1a_2a_3+a_1a_7a_8-3a_2^2a_4-7a_2^2a_8-3a_2a_3a_6+a_2a_5a_6+a_3a_4a_8,\\
C_4={}&-2a_1a_2a_8,\qquad C_5=a_1a_2a_3,\qquad C_6=a_1^2a_2.
\end{align*}
The exact map in \citet{vanRijn12var2026} is Keller and has the rational collision
\begin{equation}\label{eq:alpha}
\alpha=(0,0,-\tfrac14,0,0,0,0,0,0,0,0,0),
\end{equation}
\begin{equation}\label{eq:beta}
\beta=(1,-\tfrac32,\tfrac{13}{2},-\tfrac94,3,\tfrac32,\tfrac{99}{4},\tfrac32,-\tfrac34,-\tfrac{45}{8},-\tfrac{13}{2},\tfrac12),
\end{equation}
with $K(\alpha)=K(\beta)=\alpha$. Moreover the span of the cubic outputs has dimension six. In the displayed coordinates the first six components are already independent, so if
\[
C^{(6)}=(C_1,\ldots,C_6),\qquad B(w_1,\ldots,w_6)=(w_1,\ldots,w_6,0,\ldots,0)\in\Q^{12},
\]
the rank-compressed homogenization is the cubic-homogeneous Keller map
\begin{equation}\label{eq:G19}
G_{19}(a,w,\tau)=\bigl(a+\tau Q(a)+\tau^2B(w),\ w-C^{(6)}(a),\ \tau\bigr)\colon\Q^{19}\longrightarrow\Q^{19}.
\end{equation}
Its nonlinear part will be denoted by $h^{(19)}$. Since
\[
C^{(6)}(\alpha)=0,\qquad
C^{(6)}(\beta)=\left(-\tfrac{17}{4},-\tfrac{45}{8},\tfrac{99}{16},\tfrac92,-\tfrac{39}{4},-\tfrac32\right),
\]
the points
\begin{equation}\label{eq:rs}
r=(\alpha,0,1),\qquad s=(\beta,C^{(6)}(\beta),1)
\end{equation}
satisfy $G_{19}(r)=G_{19}(s)=r$.

Let $T_{19}$ be the symmetric trilinear polarization of $h^{(19)}$, and start the collision-generated sequence with $U_0=\operatorname{span}_{\Q}\{r,s\}$. Exact rational row reduction gives
\begin{equation}\label{eq:growth19}
\dim U_0,\ldots,\dim U_4=2,4,11,19,19.
\end{equation}

\begin{proposition}[No invariant-linear compression of the 19-variable collision]\label{prop:no19compression}
The whole space $\Q^{19}$ is the smallest linear subspace containing $r,s$ and invariant under $h^{(19)}$. The same is true after extension to $\Q(i)$ or $\C$. Consequently, 38 is minimal among examples obtained by restricting $G_{19}$ to an invariant linear subspace containing this collision and then applying the symmetric lift without further alteration.
\end{proposition}

\begin{proof}
Equation~\eqref{eq:growth19} and Theorem~\ref{thm:collision-hull} show that the collision hull is all of $\Q^{19}$. The base-change and minimality statements are then immediate from that theorem.
\end{proof}

For $X=(X_1,\ldots,X_{19})$ and $Y=(Y_1,\ldots,Y_{19})$, define
\begin{equation}\label{eq:P38}
\boxed{P_{38}(X,Y)=i\sum_{j=1}^{19}Y_jh^{(19)}_j(X+iY).}
\end{equation}
Exact collection gives 340 nonzero monomials. To record a fully explicit collision, put
\[
\rho_{38}=(I+\J h^{(19)}(s)^T)^{-1}(r-s).
\]
The exact vector is
\[
\begin{split}
\rho_{38}=(&-\tfrac{415631}{512},\ \tfrac{48861}{256},\ \tfrac{297735}{64},\ -\tfrac{2565}{64},\ \tfrac{7701}{64},\ \tfrac{4809}{64},\\
&-\tfrac{603}{16},\ \tfrac{1995}{128},\ \tfrac{147}{16},\ -\tfrac{117}{32},\ \tfrac{69}{8},\ -\tfrac{225}{8},\\
&\tfrac{417807}{512},\ -\tfrac{47421}{256},\ -\tfrac{298131}{64},\ \tfrac{2277}{64},\ -\tfrac{7077}{64},\ -\tfrac{4713}{64},\ \tfrac{13221}{64}).
\end{split}
\]
Hence the distinct points
\begin{equation}\label{eq:collision38}
A_{38}=(r,0),\qquad B_{38}=(s+\rho_{38},i\rho_{38})
\end{equation}
have the same image under $\id-\nabla P_{38}$. By Theorem~\ref{thm:collision-hull}, $P_{38}$ is Hessian-nilpotent and violates Zhao's quartic Vanishing Conjecture. As a direct low-order check, $\Delta(P_{38}^2)\ne0$, with 8,497 monomials.

\begin{proof}[Proof of Theorem~\ref{thm:main}]
The Thompson--MacFarlane assertions follow from Proposition~\ref{prop:canonical20}, Corollary~\ref{cor:min40}, and the 40-variable lift in Section~\ref{sec:forty}. The 19-variable assertions follow from Proposition~\ref{prop:no19compression} and \eqref{eq:P38}--\eqref{eq:collision38}. In each case Theorem~\ref{thm:collision-hull} gives Hessian nilpotence, the exact collision mechanism, and both vanishing statements. The monomial counts are exact collections of the displayed lift formulas.
\end{proof}

Neither route-specific minimality statement is a claim of global minimality. De Bondt and van den Essen proved the corresponding homogeneous symmetric Jacobian statement through dimension five \citep{deBondtEssen2005b}. The point of \eqref{eq:growth19} is more specific: the collision-hull method itself does not lower that particular 19-variable source map further.

\section{Exact checks}\label{sec:checks}

The companion file \texttt{anc/check\_collision\_hulls.py}, supplied as ancillary material with the arXiv version and intended as Online Resource~1 with the journal submission, contains the displayed formulas and uses only exact arithmetic over $\Q$ and $\Q(i)$. The unchanged version accompanying this manuscript has SHA-256 hash
\begin{center}
\texttt{78700e02609f9d02576bb3aa3b69c96fc0f4a7b38584a7c9056afba3bbb8c5bb}.
\end{center}
It has been run with SymPy 1.14.0. From the submission root, run
\begin{verbatim}
python3 anc/check_collision_hulls.py
\end{verbatim}
The calculation confirms:
\begin{enumerate}[label=\arabic*.]
\item cubic homogeneity of Thompson's 24-dimensional nonlinear part $H$ and the four relations in \eqref{eq:four-relations};
\item the 24-component identity $H\circ\iota=\iota\circ h$;
\item that $\iota(p),\iota(q)$ are Thompson's displayed ambient collision points and map to $\iota(p)$;
\item cubic homogeneity, term counts, and the collision for the 20-variable map $h$;
\item $(\J h)^{17}\ne0$ and $(\J h)^{18}=0$;
\item the polarization-subalgebra dimensions $2,4,11,20,20$;
\item quartic homogeneity and the 350-monomial count for $P_{40}$;
\item that $\rho_{40}$ satisfies \eqref{eq:rho40-def};
\item the distinct collision in \eqref{eq:collision40};
\item the nonvanishing and term count of $\Delta(P_{40}^2)$;
\item Van Rijn's 12-variable collision and the rank-six cubic output;
\item cubic homogeneity and the displayed collision for $G_{19}$;
\item the polarization-subalgebra dimensions $2,4,11,19,19$;
\item quartic homogeneity and the 340-monomial count for $P_{38}$;
\item the displayed $\rho_{38}$ and the collision in \eqref{eq:collision38}; and
\item the nonvanishing and term count of $\Delta(P_{38}^2)$.
\end{enumerate}
The checker independently reconstructs the two collision-generated closures using exact rational row reduction and symmetric trilinear polarization. It does not re-prove that Van Rijn's 12-variable map is Keller; that exact determinant verification is part of the independently implemented certificate accompanying \citet{vanRijn12var2026}. The general arguments and source attributions remain subject to ordinary mathematical scrutiny.

\section*{Statements and Declarations}

\paragraph{Competing interests.} The author declares no competing interests.

\paragraph{Code availability.} The exact checker described in Section~\ref{sec:checks} is supplied as ancillary material with the arXiv version and as Online Resource~1 with the journal submission.

\paragraph{Use of generative AI and AI-assisted technologies.} During the research and preparation of this work, the author used OpenAI's ChatGPT and Codex to explore candidate formulations of the collision-compression argument, organize literature, check symbolic identities, assist with drafting and typesetting, and prepare and debug the exact verification scripts. Every retained theorem-level claim is proved in the manuscript, and the listed finite symbolic claims are reproducible from the companion code. The author reviewed the resulting arguments, source attributions, text, and code and takes full responsibility for the content of the publication.

\end{document}